\documentclass[11pt,a4paper,reqno]{amsart}
\usepackage{color}
\usepackage{amssymb,amsmath,amsthm,amstext,amsfonts}
\usepackage[dvips]{graphicx}
\usepackage{psfrag}
\usepackage{url}

\makeatletter \@addtoreset{equation}{section} \makeatother

\renewcommand\thetable{\thesection.\@arabic\c@table}

\theoremstyle{plain}
\newtheorem{maintheorem}{Theorem}

\newtheorem{maincorollary}{Corollary}

\newtheorem{lemma}{Lemma}[section]

\newtheorem{remark}{Remark}[section]

\newcommand{\al} {\alpha}

\newcommand{\vep}{\varepsilon}

\begin{document}

\title{(Semi)continuity of the entropy of Sinai probability measures for partially hyperbolic diffeomorphisms}

\author{M. Carvalho}
\address{Departamento de Matem\'atica, Faculdade de Ci\^encias da Universidade do Porto\\
Rua do Campo Alegre, 687, 4069-007 Porto\\
Portugal}
\email{mpcarval@fc.up.pt}

\author{P. Varandas}
\address{Departamento de Matem\'atica, Universidade Federal da Bahia\\
 Av. Ademar de Barros s/n, 40170-110 Salvador, Brazil. \& CMUP, University of Porto, Portugal}
\email{paulo.varandas@ufba.br}
\urladdr{www.pgmat.ufba.br/varandas}

\date{\today}
\thanks{M. Carvalho has been funded by the European Regional Development Fund through the program COMPETE and by the Portuguese Government through the FCT - Funda{\c c}{\~a}o para a Ci{\^e}ncia e a Tecnologia under the project PEst-C/MAT/UI0144/2013. P. Varandas is partially supported by a CNPq-Brazil postdoctoral fellowship at University of Porto, whose research conditions are greatly acknowledged.}

\begin{abstract}
We establish sufficient conditions for the upper semicontinuity and the continuity of the entropy of Sinai probability measures invariant by partially hyperbolic diffeomorphisms and discuss their application in several examples.
\end{abstract}


\maketitle

\section{Introduction }

Given a diffeomorphism $f:M \rightarrow M$ of a smooth compact Riemannian manifold $M$, a Borelian $f$-invariant probability measure $\mu_f$ is said to be a \emph{Sinai probability measure} (as proposed in \cite[page 205]{LS83}) if its Oseledets decomposition defines unstable Pesin sub-manifolds almost everywhere along which the measure is absolutely continuous with res\-pect to Lebesgue. The probability $\mu_f$ is called \emph{Sinai-Ruelle-Bowen measure} (SRB for short) if its basin of attraction (that is, the set of points $x \in M$ such that the averages of Dirac measures along the orbit of $x$ converge to $\mu_f$ in the weak* topology) is a positive Lebesgue measure subset of the whole manifold.

The study of the continuity properties of the measure theoretical entropy goes back to Newhouse~\cite{N89} and Yomdin~\cite{Ym87}, and it is known for SRB probability measures of $C^2$ uniformly hyperbolic diffeomorphisms due to the fact that they are equilibrium states of well behaved potentials (details in \cite{Bo75}). The continuity of the SRB entropy for families of endomorphisms admitting singular points was studied in \cite{AOT06} by constructing induced maps and exploring the connection between the entropy of the SRB measure of the initial system and that of the corresponding measure of the induced system. Among conservative diffeomorphisms, the $C^1$-generic continuity of the metric entropy function has been obtained by the contribution of many authors, among which we cite \cite{BV11, ST12, YZ15}. In this setting, the continuity of the metric entropy relies on the $C^1$-generic continuity of each Lyapunov exponents and the $C^1$-generic validity of Pesin's entropy formula.

On the contrary, within the dissipative non-uniformly hyperbolic setting there is no natural reference invariant measure and the previously mentioned generic properties do not hold, or are yet to be known. For instance, generic measures of $C^1$-generic diffeomorphisms do not satisfy the Pesin's entropy formula (see~\cite{AbBnCr11}). In addition, it has been conjectured by J. Palis that in the complement of uniform hyperbolicity (a $C^1$-open property) every diffeomorphism may be approximated by another exhibiting either a homoclinic tangency or a heteroclinic cycle, a phenomenon whose generic unfolding may generate infinitely many sinks or sources. Thus, although the existence of a Sinai probability measure has been conjectured to hold generically, its uniqueness may be, in some sense, a rare event. For all these reasons, we will analyze the continuous dependence of the Sinai probability measure and its entropy in cases where it happens to exist and to be unique.

After L.S. Young pioneering work \cite{LSY98}, the existence of Markov towers for dynamical systems has become a suitable machinery to deduce fine statistical properties of invariant measures; see, for instance \cite{AOT06}. The construction of Markov towers is a hard subject, mostly understood for non-uniformly hyperbolic dynamical systems (the ones without zero Lyapunov exponents). Proposing a complementary approach to attain continuity of the metric entropy map with respect to the dynamics, we will demand the existence of a dominated splitting that is adapted to the Sinai probability measure, although no request is made concerning the absence of zero Lyapunov exponents.

\section{Setting}

Consider $r \in [1,+\infty[$ and let $\text{\text{Diff}}^{\,r}(M)$ denote the set of $C^{\,r}$ diffeomorphisms, endowed with the $C^r$ topology, of a smooth compact Riemannian manifold $M$ to itself.

\subsection{Dominated splittings} 
Given $f \in \text{Diff}^{\,1}(M)$ and a compact $f$-invariant set $\Lambda\subset M$, we say that $\Lambda$ admits a \emph{dominated splitting} if there exists a $Df$-invariant decomposition $T_\Lambda M= E \oplus F$ and constants $C>0$ and $\varrho \in (0,1)$ such that
$$\forall n \in \mathbb{N} \quad \| Df^n(x) \mid_{E_x} \| \,.\, \| (Df^n(x) \mid_{F_x})^{-1} \|^{-1} \le C\varrho^n.$$
Recall that a dominated splitting (and, in particular, the dimension of its spaces) persists under small $C^1$ perturbations and varies continuously with the base point and the dynamics; details in \cite{BDV05}.

\subsection{Partial hyperbolicity} A set $\Lambda$ is said to be \emph{strongly partially hyperbolic} if the tangent bundle over $\Lambda$ splits as a direct sum of continuous sub-bundles $E^u\oplus E^c \oplus E^s$, invariant under the derivative $Df$ and such that:
\begin{itemize}
\item at least two of the three sub-bundles are nontrivial (that is, both have positive dimension);
\item $Df_{|E^u}$ is uniformly expanding and $Df_{|E^s}$ is uniformly contracting;
\item $Df_{|E^c}$ is never as expanding as $Df_{|E^u}$ nor as contracting as $Df_{|E^s}$ (although its behavior may vary from point to point); that is, the $Df$-invariant decompositions $E^u \oplus E^c$ and $E^c \oplus E^s$ are dominated.
\end{itemize}
$\Lambda$ is hyperbolic precisely when the central bundle $E^c$ may be taken trivial. The set $\Lambda$ is said to be \emph{partially hyperbolic} if the tangent bundle over $\Lambda$ has a dominated splitting $E\oplus F$, invariant under the derivative $Df$ and such that either $E$ or $F$ is hyperbolic (that is, uniformly expanding or uniformly contracting, respectively). The diffeomorphism $f$ is called partially hyperbolic if $M$ itself is a partially hyperbolic set.

\subsection{Robust transitivity} A diffeomorphism $f$ is \emph{transitive} if there exists a point $x \in M$ whose forward orbit $\{f^n(x): n \in \mathbb{N}\}$ is dense in $M$. It is \emph{$C^1$-robustly transitive} if there exists a $C^1$ neighborhood $U$ of $f$ such that each $g \in U$ is transitive. In \cite{M78}, R. Ma\~n\'e proved that, in dimension $2$, $C^1$-robust transitivity implies hyperbolicity (so Anosov diffeomorphisms are the only robustly transitive ones) and, moreover, that this implication no longer holds in higher dimensions. Yet, in \cite{DPU99} it has been proved that, in dimension $3$, every robustly transitive diffeomorphism is partially hyperbolic. Although this characterization do not carry over to dimensions bigger or equal to $4$ (more information may be found in \cite{BV00}), it was proved in \cite{BDP03} that a $C^1$ robustly transitive diffeomorphism always admits a dominated splitting.

\subsection{Sinai probability measures}\label{Sinai} Among $C^{1+\alpha}$ uniformly hyperbolic diffeomorphisms, the property that characterizes a SRB measure is linked to the absolute continuity of the measure with respect to Lebesgue along unstable manifolds; see \cite{Bo75}. This last property is a necessary and sufficient condition, as proved in \cite{LY85}, for the metric entropy of the Sinai probability measure, say $\mu_f$, to be computed as the integral of the positive Lyapunov exponents, that is,
\begin{equation}\label{eq:Pesinformula}
h_{\mu_f}(f)
	= \int \big(\sum_{\lambda_i(x,f) >0} \lambda_i(x,f) \big)\, d\mu_f
\end{equation}
where $\lambda_1 (x,f)\geq \lambda_2 (x,f)\geq \cdots \geq \lambda_{\text{dim }M} (x,f)$ denote all Lyapunov characteristic exponents of $\mu_f$ at $x$. We will refer to $\mu_f$ as a \emph{Sinai probability measure} whenever $\mu_f$ is absolute continuous with respect to Lebesgue along unstable manifolds associated to its Oseledets-Pesin decomposition.

Given a $C^{1+\alpha}$ diffeomorphism $f$, an open $\mathcal{W} \subseteq M$ and a transitive hyperbolic attractor $\Lambda_f:=\bigcap_{n\ge 0} f^n(\mathcal W)$ of $f$, the entropy of the natural Sinai probability measure $\mu_f$ for $f_{\mid_{\Lambda_f}}$ is well known to exist and vary continuously with $f$ since it is the equilibrium state for $f$ with respect to the potential
$$x \in \Lambda_f \mapsto -\log |\det Df(x)\mid_{E^u_x}|.$$
For $C^1$-diffeomorphisms Sinai probability measures may not exist even in the uniformly hyperbolic setting, as hinted in \cite{AB06, AB07}.

\section{Main results}

As proved in \cite{AbBnCr11}, generic (with respect to the weak* topology) ergodic probability measures invariant by $C^1$-generic diffeomorphisms are hyperbolic (that is, they exhibit no zero Lyapunov exponents) and have zero metric entropy. In particular, this property implies that $C^1$-generically the generic measures do not satisfy Pesin's formula, unless all their Lyapunov exponents are negative. Moreover, recent contributions \cite{AB06, AB07} suggest that, in the dissipative setting, $C^1$ generically the existent SRB probability measures may not be absolutely continuous with respect to Lebesgue measure along unstable Pesin foliations, even in the uniformly hyperbolic context. Additionally, Pesin's entropy formula may fail in the $C^1$-topology. Accordingly, we will restrict our study to the set of $C^{r}$ diffeomorphisms ($r \geq 1$) with a unique Sinai probability measure.

\subsection{Upper semicontinuity} We will start showing that if for a family of $C^1$ diffeomorphisms Sinai probability measures exist, are unique and vary continuously, then the corresponding measure theoretical entropy map is upper semicontinuous.

\begin{maintheorem}\label{thm:A}
Let $\mathcal U \subset \text{\text{Diff}}^{\,\,r}(M)$, $r\geq 1$, be a Baire set of $C^{r}$ diffeomorphisms of $M$ such that, for every $f \in \mathcal U$, there exists a unique Sinai probability measure $\mu_f$ which varies continuously with $f$. Then the entropy map $h_{S}: \mathcal U \to \mathbb R_0^+$ given by
\begin{equation}\label{eq:entropy}
f \in \mathcal U \quad \mapsto  \quad h_{\mu_f}(f)
\end{equation}
is upper semicontinuous. In particular, there exists a $C^r$-residual subset $\mathcal R \subset \mathcal U$ of continuity points of the function $h_{S}$.
\end{maintheorem}

\begin{proof}[Proof of Theorem~\ref{thm:A}]

It follows from \cite{LS83} that, as $\mu_f$ is a Sinai probability measure, then
\begin{equation}\label{LedSt}
h_{\mu_f} (f) = \lim_{n\to\infty}  \frac1n  \int \log \| Df^n(x)^{\wedge}\| \, d\mu_f
\end{equation}
where
$$\| Df^n(x)^{\wedge}\|:=1+\sum_{j=1}^{\dim M} \| Df^n(x)^{\wedge j}\|$$
and $Df^n(x)^{\wedge j}$ denotes the $j$th exterior power of the linear map $Df^n(x)$. Moreover, it is clear that the map
$$ f \in \mathcal U \quad \mapsto \quad \frac1n  \int \log \| Df^n(x)^{\wedge}\| \, d\mu_f $$
is continuous. Besides, it is not hard to check that, given $f$, the sequence of real functions $(\log \| Df^n(x)^{\wedge}\|)_{n \in \mathbb{N}}$ is sub-additive. Indeed, for every $m,n\ge 1$ and $x\in M$,
\begin{align*}
\log \| Df^{n+m}(x)^{\wedge}\|
	& \le \log \Big( 1+\!\!\sum_{j=1}^{\dim M} \! \| Df^n(x)^{\wedge j}\| \, \| Df^m(f^n(x))^{\wedge j}\| \Big) \\
	& \le \log
		\Big[
		\Big( 1+ \!\!\sum_{j=1}^{\dim M} \! \| Df^n(x)^{\wedge j}\|  \Big)
		\Big( 1+ \!\!\sum_{j=1}^{\dim M} \! \| Df^m(f^n(x))^{\wedge j}\|  \Big)
		\Big] \\
	& = \log \| Df^{n}(x)^{\wedge}\| + \log \| Df^{m}(f^n(x))^{\wedge}\|.
\end{align*}
Therefore, the measure theoretical entropy of $\mu_f$ is given by
$$h_{\mu_f} (f) = \inf_{n \ge 1}  \frac1n  \int \log \| Df^n(x)^{\wedge}\| \, d\mu_f $$
and so it is the greatest lower bound of a sequence of real continuous functions on $\mathcal U$. Therefore, the map
$$f \in \mathcal U \quad \mapsto \quad  h_{\mu_f}(f)$$
is upper semicontinuous, which proves the first assertion in the theorem. The second assertion is a direct consequence of the first one
by general topology arguments.
\end{proof}

Some comments are in order. Firstly, observe that the uniqueness assumption in the theorem was not strictly used. In particular, the upper semicontinuity of the entropy map $f \mapsto h_{\mu_f}(f)$ holds for any given family
of Sinai probability measures that varies continuously with the dynamics in the weak* topology. A second comment concerns the differentiability assumption on the theorem. As stated, given a family of $C^r$ diffeomorphism parameterized by some Baire space, the theorem yields a $C^r$-generic set of points of continuity for the metric entropy map. The set of continuity points is clearly $C^r$-dense, hence $C^1$-dense as well.

\begin{remark}\label{rmk:uniformity}
\emph{Equation \eqref{LedSt}, which is the key characterization used in the previous argument, has been established in \cite{LS83} also for $C^{1+\alpha}$ endomorphisms $f$ with a singular
set $\mathcal S_f$ satisfying the following conditions:
\begin{itemize}
\item[(LS1)] If $B_\vep (\mathcal S_f)$ stands for the $\vep$-ball around $\mathcal S_f$, there are uniform constants $C_{f,1}$ and $\beta>0$ such that, for every small enough $\vep>0$,
$$\mu_f(B_\vep (\mathcal S_f)) \le C_{f,1}\,\, \vep^\beta.$$
\item[(LS2)] There exists $C_{f,2}>0$ such that
$$\int \log^+ \| Df(x)^\pm\| \, d\mu_f \le C_{f,2} < \infty$$
where $\log^+ (t)=\max \{\log(t),0 \}$.
\end{itemize}
Therefore, Theorem~\ref{thm:A} is valid in the more general context of $C^{1+\alpha}$ local diffeomorphisms or $C^{1+\alpha}$ endomorphisms with singular sets $\mathcal S_f$ satisfying the conditions (LS1) and (LS2) and for which\\
\begin{itemize}
\item[(LS3)] $\forall\,\, n \in \mathbb{N}, \quad f \mapsto \int \log \| Df^n(x)^{\wedge}\| \, d\mu_f \quad \text{is continuous}.$
\end{itemize}
\medskip
We shall discuss this setting later, in Section~\ref{examples}.}
\end{remark}

\subsection{Conservative case revisited}\label{subsec:conserv}

Let $M$ be a smooth compact Riemannian manifold, $m$ be a normalized Lebesgue measure on $M$ and $\text{Diff}^1_m(M)$ be the set of $C^1$ volume-preserving diffeomorphisms of $M$ endowed with the $C^1$ topology. In this setting, we may gather signi\-ficant information concerning dominated splittings, the metric entropy of $m$, its Lyapunov spectrum and the validity of Pesin's entropy formula.

As proved in  \cite{ST12} and \cite{T02}, Pesin's entropy formula holds $C^1$ generically among volume-preserving diffeomorphisms. Moreover, according to \cite{BV05}, there is a residual subset of $\text{Diff}^1_m(M)$ such that, for any $f$ in that set and $m$-almost every point $x$, either all Lyapunov exponents at $x$ are zero or the Oseledets splitting of $f$ is dominated on the orbit of $x$. In particular, see \cite{B02}, for any compact surface $M$ there is a residual set of area preserving diffeomorphisms which are either Anosov or have zero Lyapunov exponents $m$ almost everywhere. The proofs of these results indicate that, in the conservative setting, the entropy map $f \in \text{Diff}^1_m(M) \rightarrow h_m(f)$ is not continuous. Nevertheless, it has been established recently in \cite{AB12} that the Lyapunov spectrum is $C^1$-generically continuous; together with the $C^1$-generic validity of Pesin's entropy formula, this implies that there is a residual subset of $\text{Diff}^1_m(M)$ of continuity points of the entropy map; see \cite{YZ15}.

From these results and \cite{BV05} it is now easy to deduce, after intersecting the corresponding residual sets, that, if $f \in \text{Diff}^{\,1}_m(M)$ is a generic continuity point of the entropy map, then:
\begin{itemize}
\item[(a)] If $M$ is the $2$-torus, then $f$ is either Anosov or $h_{m}(f)=0$.
\item[(b)] If $M$ is a compact surface other than the $2$-torus, then $h_{m}(f)=0$.
\item[(c)] If $\text{dim }M > 2$, then either $f$ has a dominated splitting or $h_{m}(f)=0$.
\end{itemize}

We recall that the set $\text{Diff}^{\,\,1}_m(M)$ is closed but has empty interior in $\text{Diff}^{\,1}(M)$. Moreover, the probability $m$, although preserved by all elements in $\text{Diff}^1_m(M)$, may not be a Sinai probability measure. Still, according to \cite{AbBnCr11}, inside the family of robustly ergodic conservative diffeomorphisms there is an open dense subset $\mathcal V$ where $m$ is hyperbolic with respect to any diffeomorphism in $\mathcal V$ (that is, its Lyapunov exponents are all non-zero) and its Oseledets decomposition is dominated. Additionally, by \cite{BFP06}, if $f$ belongs to $\text{Diff}_m^{1+\alpha}(M)$ and is $C^1$ robustly ergodic (that is, $m$ is ergodic for $f$ and for any $C^1$ near-by element of $\text{Diff}_m^{1+\alpha}(M)$), then $f$ may be $C^1$ approximated by $g \in \text{Diff}_m^{1+\alpha}(M)$ for which $m$ is hyperbolic. These results motivate the following application of Theorem~\ref{thm:A} to $C^r$ conservative diffeomorphisms.

\begin{maincorollary}
Assume that $\mathcal U \subset \text{Diff}^{\,\,r}_m(M)$, $r \geq 1$, is a Baire space and that $m$ is a Sinai probability measure for any $f\in \mathcal U$. Then the entropy function
$$h_m: f \in \mathcal U \mapsto h_m(f)$$
is upper semicontinuous. Consequently, there exists a $C^r$ residual subset $\mathcal R\subset \mathcal U$ of continuity points of $h_m$.
\end{maincorollary}

\subsection{Continuity}

In the $C^{1+\alpha}$ uniformly hyperbolic setting, Sinai probability measures are equilibrium states of an H\"older continuous potential, namely
$$ x  \quad \mapsto \quad \varphi^u(x)= -\log | \det Df(x) \mid_{E^u_x} |,$$
and so the metric entropy of such Gibbs measures is given by
$$h_{\mu_f}(f)= \int \, - \varphi^u(x)\, d\mu_f.$$
In general, a similar formula also holds in $\text{Diff}^{1+\alpha}(M)$, that is, $\mu_f$ has an adequate Jacobian as well. This is the key to prove the following property.

\begin{maintheorem}\label{thm:B}
Given $\alpha > 0$, assume that  $\mathcal{U}$ is a Baire subset of $\text{Diff}^{\,\, r}(M)$, $r\ge 1+\alpha$, such that every $f\in \mathcal U$ has a $Df$-invariant dominated splitting $TM=E_f \oplus F_f$
and a Sinai probability measure $\mu_f$ which has exactly $\dim F_f$ non-negative Lyapunov exponents.
Then:
\begin{itemize}
\item[(a)] If the Sinai probability measure is unique and the entropy map
\begin{equation}\label{eq:entropy}
h_{S}: f \in \mathcal U \quad \to \quad h_{\mu_f}(f)
\end{equation}
is upper semicontinuous, then it is continuous.
\item[(b)] If the Sinai probability measures vary continuously with the dynamics, then $h_{S}$ is continuous.
\end{itemize}
\end{maintheorem}

\begin{proof}[Proof of Theorem~\ref{thm:B}]
The proof of this theorem makes use of the regularity of the Jacobian along the $F$ sub-bundle as we now describe. As a dominated splitting and the dimensions of the corresponding subspaces vary continuously with the base point \cite{BDV05}, the map $x \in M \mapsto F_{f,x}$
is continuous and so the Jacobian
$$x \in M \mapsto Jf_{F_{f,x}}(x) := |\det Df(x)\mid_{F_{f,x}}|$$
is continuous as well. Then, from the proof of Proposition~2.5 in \cite{LS83}, we get
\begin{equation}\label{eq:LedS}
\int \big( \sum_{\lambda_i(x) \ge 0} \lambda_i(x) \Big)\, d\mu_f
	= \int \log Jf_{F_{f,x}}(x) \, d\mu_f.
\end{equation}
In particular, the continuity of the entropy function in \eqref{eq:entropy} is now a direct consequence of the continuity of the previous expression with the dynamical system $f$. Thus, we are left to prove this last assertion.

Firstly, for each $f\in \mathcal U$, let $TM=E_f \oplus F_g$ be the corresponding dominated splitting. Our assumptions ensure that the map
\begin{equation}\label{eq:teoB.cont}
 (f,x) \in \mathcal U \times M  \quad \mapsto \quad Jf_{F_{f,x}}(s):= |\det Df(x) \mid_{F_{f,x}}|
\end{equation}
is continuous. \\

\noindent(a) Assume that the entropy map is upper semicontinuous (a property valid if $\mathcal{U} \subset \text{Diff}^{\,\, \infty}(M)$; see \cite{N89, Ym87}) and that the Sinai probability measure is unique for each $f\in \mathcal U$. Then:

\begin{lemma} The map
$$f \in \mathcal U \quad \mapsto \quad  \mu_f$$
is continuous in the weak* topology.
\end{lemma}

\begin{proof}
Consider a sequence $(f_n)_{n \in \mathbb{N}}$ in $\mathcal U$ converging to $f \in \mathcal U$ and (taking a subsequence if necessary) such that $(\mu_{f_n})_{n \in \mathbb{N}}$ converges in the weak* topology to $\nu$. Therefore:\\

\noindent (i) Given the property (\ref{eq:teoB.cont}) and the fact that $\nu$ is the weak* limit of $\mu_{f_n}$, we have
$$\lim_{n \to +\infty}\,\int \log Jf_{{n}_{F_{f_n,x}}}(x) \, d\mu_{f_n} = \int \log Jf_{F_{f,x}}(x) \, d\nu.$$

\noindent (2i) As $\mu_{f_n}$ is a Sinai probability measure for each $n$,
$$\limsup_{n \to +\infty}\,\, h_{\mu_{f_n}}(g_n) = \lim_{n \to +\infty}\,\, \int \log Jf_{{n}_{F_{f_n,x}}}(x) \, d\mu_{f_n}.$$

\noindent (3i) By assumption, we know that $\limsup_{n \to +\infty}\,\, h_{\mu_{f_n}}(f_n) \leq h_\nu(f)$. Hence,
$$\int \log Jf_{F_{f,x}}(x) \, d\nu \leq h_\nu(f).$$

\noindent (4i) From from Proposition~2.5 in \cite{LS83}, which is valid for any invariant probability measure, we get
$$\int \big( \sum_{\lambda_i(x) >0} \lambda_i(x) \Big)\, d\nu = \int \log Jf_{F_{f,x}}(x) \, d\nu,$$
so
$$\int \big( \sum_{\lambda_i(x) >0} \lambda_i(x) \Big)\, d\nu \leq h_\nu(f).$$

\noindent (5i) Finally, by Ruelle's inequality \cite{Ru78},
$$h_\nu(f) \leq \int \big( \sum_{\lambda_i(x) >0} \lambda_i(x) \Big)\, d\nu$$
and so $\nu$ is a Sinai probability measure. As $f$ has a unique Sinai probability measure, $\nu$ must be $\mu_f$.
\end{proof}

Consequently, the right hand side of \eqref{eq:LedS} varies continuously with the dynamics. Thus,
$$f \in \mathcal U  \quad \mapsto \quad  h_{\mu_f}(f)= \int \log Jf_{F_{f,x}}(x) \, d\mu_f$$
is a continuous function, and this ends the proof of Theorem~\ref{thm:B}~(a).\\

\noindent (b) By assumption, the Sinai probability measures vary continuously with the dynamics in $\mathcal U$, thus, by (\ref{eq:teoB.cont}), the function
$$f \in \mathcal U  \quad \mapsto \quad  \int \log Jf_{F_{f,x}}(x) \, d\mu_f$$
is continuous. Moreover, for any $f \in \mathcal{U}$,
$$ h_{\mu_f}(f)= \int \log Jf_{F_{f,x}}(x) \, d\mu_f$$
so the entropy varies continuously with the dynamics.
\end{proof}

\begin{remark}\label{hyp.times}
\emph{The previous theorem can be strengthened in several ways. Firstly, the full assumption on domination is not essential for the proof of Theorem B; we just need a $Df$-invariant continuous sub-bundle $F_f \subset TM$ so that the Jacobian $x \to \log Jf_{F_{f,x}}$ is continuous and varies continuously with $f$. In Subsection~\ref{het.bif} we will illustrate that the latter regularity may hold for diffeomorphisms with no dominated splitting, although the set of points with lack of domination must be a meager set. Secondly, the assumption on the continuity of the Sinai measures can be replaced by the uniqueness plus the assumption of the existence of a uniform $\delta> 0$ such that every partition $\mathcal{P}$ with diameter at most $\delta$ is generating for every $\mu_f$ with $f \in \mathcal{U}$. This last condition is valid, for instance, if each $\mu_f$ has infinitely many $(\sigma_{\mu_f}, \delta)$-hyperbolic times almost everywhere, for some $\sigma_{\mu_f}> 1$ and uniform $\delta>0$. Following \cite{V07}, this condition implies the upper semicontinuity of the measure theoretical entropy, and ultimately the
continuity of the (unique) Sinai measure of the dynamical system.}
\end{remark}

\begin{remark} \emph{One says that a $f$-invariant compact set $\Lambda \subseteq M$ is a \emph{robustly transitive attractor} if there exist a neighborhood $U$ of $\Lambda$ and a $C^1$ neighborhood $\mathcal N$ of $f$ such that
\begin{itemize}
\item $\Lambda$ is the maximal invariant set of $f$ in $U$;
\item for every $g \in \mathcal N$, the maximal invariant set $\Lambda_g=\bigcap_{n\in\mathbb{N}}\, g^n(U)$ of $g$ in $U$ contains dense orbits.
\end{itemize}
The set $\Lambda$ is said to be a \emph{partially hyperbolic attractor} if there exists a dominated splitting $T_\Lambda M= E^s \oplus E^{cu}$ and constants $C>0$ and
$\varrho \in (0,1)$ such that
\begin{itemize}
\item  $\|Df^n(x)\mid_{E_x^s}\| \le C\varrho^n$ (uniform contraction along $E^s$);
\item $E^{cu}$ is non-uniformly expanding, that is,
	$$	\limsup_{n\to\infty} \frac1n \sum_{j=0}^{n-1} \log \| (Df(f^j(x)) \mid_{F_{f^j(x)}})^{-1} \| < 0. $$
\end{itemize}
We note that the previous theorem is valid for invariant compact sets $\Lambda$ which are robustly transitive and partially hyperbolic attractors.}
\end{remark}

\subsection{Regularity of the Lyapunov spectrum}
Given a $C^{1+\alpha}$ diffeomorphism $f$, an open $\mathcal{W} \subseteq M$ and a transitive hyperbolic attractor $\Lambda_f:=\bigcap_{n\ge 0} f^n(\mathcal W)$ of $f$, the entropy of the natural Sinai probability measure $\mu_f$ for $f_{\mid_{\Lambda_f}}$ is well known to exist and vary continuously with $f$.
Consequently, we also have regularity of the Lyapunov exponents in this context.

\begin{maintheorem}\label{thm:continuityexp}
Let $\mathcal U \subset \text{Diff}^{\,\, r}(M)$, where $r\ge 1+\alpha$ and $\alpha>0$, be an open set of
transitive Anosov diffeomorphisms. There exists a residual subset $\mathcal R \subset \mathcal U$ such that every $f\in \mathcal R$
is a continuity point of the Lyapunov exponent function
$$f \in \mathcal U  \mapsto \mathcal{L}_i(f):= \inf_{n\ge 1} \frac1n \int \log \| Df^n(x)^{\wedge i}\| \, d\mu_f$$
for every $1\le i \le \dim M$.
\end{maintheorem}

\begin{proof}
Since the Sinai probability measure is unique for every $f\in \mathcal U$ and varies continuously with $f$, it follows from the sub-additivity of the sequence of functions $(\log \| Df^n(x)^{\wedge i}\|)_{n \in \mathbb{N}}$ that the function $\mathcal U \ni f \mapsto \mathcal{L}_i(f)$ is upper semicontinuous, for any $1\le i \le \dim M$. Thus, there exists a residual subset $\mathcal R_i \subset \mathcal U$ of continuity points for $\mathcal{L}_i$. We are left to intersect these residual subsets, obtaining $\mathcal R:= \bigcap_{1\le i \le \dim M}\mathcal R_i$.
\end{proof}

\subsection{Flows}

For volume preserving flows, the $C^1$-generic continuity of the metric entropy function has been obtained in the three-dimensional setting in \cite{BV11} by proving the $C^1$-generic continuity of each Lyapunov exponent and the $C^1$-generic validity of Pesin's entropy formula; it remains an open question in the higher dimensional setting.

Given a vector field $X\in \mathfrak{X}^r(M)$ $(r\ge 1$) on a compact Riemannian manifold $M$, it generates a $C^r$-smooth flow $(\varphi_X^t)_{t\in \mathbb R}$. Moreover, for any invariant probability measure its Lyapunov exponents with respect to the flow $(\varphi_X^t)_{t\in \mathbb R}$ coincide with the ones for the time-one $C^r$-diffeomorphism $\varphi_X^1$. The measure theoretical entropy of the flow is also defined as the entropy of $\varphi_X^1$. Therefore, the extension of Theorem~\ref{thm:A} for flows is straightforward.

Clearly, if $TM=E\oplus F$ is a dominated splitting for the flow $(\varphi_X^t)_{t\in \mathbb R}$, then it is also a dominated splitting for the time-one map $\varphi_X^1$. Moreover, since dominated splittings vary
continuously with the base point, the one-dimensional $D\varphi_X^1$-invariant subspace generated by the vector field $X$ is contained in one of the subbundles $E$ or $F$ and gives rise to a zero Lyapunov exponent. Hence, the existence of a dominated splitting for the flow yields a dominated splitting for $\varphi_X^1$ and gives rise to a natural counterpart of Theorem~\ref{thm:B} for flows. 
As an illustration let us observe that geometric Lorenz attractors $\Lambda$ for  
three-dimensional manifolds admit a unique Sinai measure (see~\cite{APPV}) which varies continuously 
with the vector field (see~\cite{AlS}). Since there exists an invariant splitting $T_\Lambda M=E^s\oplus F$ on the
attractor $\Lambda$, which varies continuously with the vector field and there are $\dim F=2$ non-negative
Lyapunov exponents for the flow then the counterpart of Theorem~\ref{thm:B} in this setting yields that the entropy of the Sinai measure varies continuously with the vector field.

\section{Examples}\label{examples}

\subsection{Diffeomorphisms derived from Anosov}
A particularly well studied family of partially hyperbolic diffeomorphisms is the family of $C^{1+\alpha}$ derived from Anosov diffeomorphisms; see \cite{M78, Ca93, Cas04}. These systems are not uniformly hyperbolic neither structurally stable but they are robustly transitive and intrinsically stable (see \cite{BuFSV12} for details; their topological entropy is in fact constant in a neighborhood and they each have a unique measure of maximal entropy with respect to which periodic orbits are equidistributed). Moreover each has a unique SRB probability measure which is an equilibrium state of an H\"older potential and absolutely continuous with respect to Lebesgue along the unstable Pesin sub-manifolds. Therefore, we may apply either Theorem~\ref{thm:B}~(a) or Theorem~\ref{thm:B}~(b) according to the arcs considered, thus deducing that the metric entropy of the Sinai probability measure varies continuously within these dynamics.

\subsection{Heteroclinic bifurcations of Anosov diffeomorphisms}\label{het.bif}
In \cite{E98}, another family of maps at the boundary of the set of the Anosov diffeomorphisms has been studied from an ergodic viewpoint. Given a $C^2$ transitive Anosov diffeomorphism $f_0:M \rightarrow M$, the new dynamics $f_1$ is obtained by isotopy from $f_0$ as a first bifurcation through a cubic tangency between the stable and unstable manifolds of two distinct periodic points of $f_0$. The author proves that $f_1$ is conjugate to $f_0$, so it is transitive as well, and has a unique Sinai probability measure $\mu_1$ with respect to which $f_1$ is Bernoulli and has a Pesin region with full measure. The key property of these bifurcations is that the invariant stable and unstable foliations persist in $M$, although at the tangency point they are no longer transversal. Moreover, if the dimension of $M$ is two and $f_0$ is conservative, then the isotopy may be taken in the space of conservative diffeomorphisms and $\mu_1$ is the Lebesgue measure. These family does not satisfy the requirements of Theorem~\ref{thm:B} due to the point of heteroclinic tangency. Nevertheless, the unstable sub-bundle extends continuously to that point and so Theorem~\ref{thm:B}~(b) applies to this setting.

\subsection{Mostly expanding or mostly contracting diffeomorphisms}
For $C^1$-perturbations of a diffeomorphism that admit a unique Sinai probability measure continuously parametrized by a Baire space, Theorem~\ref{thm:A} ensures that $C^1$-generically the measure theoretical entropy of the Sinai probability measure varies continuously. For instance, we may consider the $C^1$-perturbations of mostly expanding or mostly contracting  partially hyperbolic diffeomorphisms, introduced in \cite{ABV00} and \cite{BV00}, respectively. Moreover, we note that, under the assumption of robust transitivity, it was proved in \cite{ABV00, BV00} that, in this context, there exists a unique SRB probability measure which is absolutely continuous with respect to Lebesgue measure along unstable Pesin manifolds, it is a $u$-Gibbs measure \cite{PeS83} and its statistical stability was obtained in \cite{V07}. Consequently,

\begin{maincorollary}
The entropy of the SRB measure for mostly contracting or mostly expanding diffeomorphisms varies continuously with the dynamical system.
\end{maincorollary}

\subsection{Maps with singularities and discontinuities}
A robust class $\mathcal{U}$ of multidimensional non-uniformly expanding $C^2$ endomorphisms, now called Viana maps, were introduced in \cite{Vi97} as perturbations of skew products of fiber quadratic maps. Some extensions have been studied more recently in \cite{AS13,Va14} and references therein. These are transitive maps with a zero Lebesgue meassure singular set $S_f$ and a unique Sinai probability measure $\mu_f$ which varies continuously in the weak* topology (and even in the $L^1$ topology; details in \cite{AV02}). Using Markov structures, that exist and vary continuously with the dynamics in $\mathcal{U}$, it was proved in \cite{AOT06} that the entropy of the Sinai probability measure also varies continuously with the endomorphism. What information may we get in this setting from Theorem~\ref{thm:A}?

Two of the demands made in Remark~\ref{rmk:uniformity}, needed to apply Theorem~\ref{thm:A}, pose no problem. Indeed:
\begin{itemize}
\item As $\mu_f$ has positive Lyapunov exponents in every direction and almost every point, by \cite{L98, QZ02} the Pesin's entropy formula is valid, that is,
$$h_{\mu_f}(f) = \int \log |\det Df| \, d\mu_f.$$
So, instead of condition (LS3) of Remark~\ref{rmk:uniformity}, we are left to verify that, in spite of the singular region, the map
$$f \mapsto \log |\det Df|$$
is continuous. This demand is fulfilled within well chosen subfamilies of $\mathcal{U}$.
\item In $\mathcal{U}$, the Lyapunov exponents are uniformly bounded from above and below, so there exists a uniform constant $A>0$ such that, for any $f \in \mathcal{U}$, we have
$$(\mathcal{D}) \quad \quad |\int \log |\det Df| \, d\mu_f| \le A.$$
This means, in particular, that condition (LS2) of Remark~\ref{rmk:uniformity} holds.
\end{itemize}

\medskip

However, condition (LS1) of Remark~\ref{rmk:uniformity}, which is linked to a slow rate of approximation of the orbits to the singular set, is not guaranteed for Viana maps.

Yet, Viana maps satisfy another property which is enough to prove directly the continuity of the entropy with the dynamics. Namely, it is possible to choose adequately the skew product of the fiber quadratic maps, of whose perturbations is made $\mathcal{U}$, in order to produce a parameterized $C^3$ small enough arc
$$(f_t)_{t\,\in\, [-\rho_0,\,\rho_0]}$$
of $C^2$ endomorphisms in $\mathcal{U}$, exhibiting singular sets
$$\mathcal S_{f_t}=\mathcal S$$
with $\mu_{f_t}(\mathcal S)=0$ and such that there exist constants $c>0$ and $\beta > 0$ satisfying, for every $t,s \in [-\rho_0,\,\rho_0]$ and all $x \in M\setminus \mathcal S$,
$$(\mathcal{E})\quad \quad |\,\log |\det Df_t(x)| - \log |\det Df_s(x)|\,|\,  \le \,c \,|t-s|^\beta.$$
Taking into account that $\mu_{f_t}$ varies continuously in the $L^1$ topology with $t$, that the measure $\mu_{f_t}$ is regular and that $\mu_{f_t}(S)=0$ for any $t$, given $\epsilon>0$ there is an open neighborhood of $S$ in $M$, say a ball $B_{\delta}(S)$ for some $\delta=\delta(\epsilon) >0$ independent of $t \in [-\rho_1,\,\rho_1] \subset [-\rho_0,\,\rho_0]$, such that, for every $t$,
$$\mu_{f_t}(B_{\delta}(S))< \epsilon/2.$$
By condition $(\mathcal{D})$, the map $\log |\det Df_t  | \in L^1(\mu_{f_t})$ and, consequently, from property $(\mathcal{E})$ we get
\begin{align*}
|\,\int_{B_\delta(\mathcal S)} \log |\det Df_t  | \;d\mu_{f_t}\,|
	\lesssim \,c \,|t|^\beta + \int_{B_\delta(\mathcal S)} \log |\det Df_0  | \;d\mu_{f_t}.
\end{align*}
Additionally, for any fixed $t$, the Dominated Convergence Theorem yields
$$\lim_{\delta\to 0} \int_{B_\delta(\mathcal S)} \log |\det Df_0  | \;d\mu_{f_t} = 0.$$
This proves that
$$ \limsup_{\delta \to 0}\,\, \limsup_{t\to 0} \int_{B_\delta(\mathcal S)} \log |\det Df_t  | \;d\mu_{f_t} = 0.$$
Hence, as
\begin{align*}
h_{\mu_{f_t}}(f_t)
	& = \int_{B_\delta(\mathcal S)} \log |\det Df_t | \;d\mu_{f_t}
		+ \int_{M\setminus B_\delta(\mathcal S)} \log |\det Df_t  | \;d\mu_{f_t}
\end{align*}
we finally conclude from the weak* convergence of the measures that
$$\lim_{\delta\to 0}\,\,\lim_{t\to 0}
	\int_{M\setminus B_\delta(\mathcal S)} \log |\det Df_t  | \; d\mu_{f_t}
	= \int_{M} \log |\det Df_0  | \;d\mu_{f_0}
	= h_{\mu_{f_0}}(f_0).
$$
Therefore,
$$\lim_{t\to 0}\,\,h_{\mu_t}(f_t) = h_{\mu_0}(f_0).$$

\begin{remark}\emph{Condition $(\mathcal{E})$ has been used, among other assumptions, in \cite{A07} to provide an alternative proof for the continuity (in the weak* topology) of the equilibrium states of the Viana maps. Property  $(\mathcal{E})$ also holds for the family of Benedicks-Carleson quadratic maps.}
\end{remark}

\subsection{Intermittency phenomenon for interval maps and diffeomorphisms}
We observe that the statement of Theorem~\ref{thm:B} does not require, \emph{a priori}, neither hyperbolicity of the Sinai probability measures nor their continuity with the dynamics. We can illustrate this detail by the so-called Manneville-Pomeau transformations \cite{LSV99} and the setting of almost Anosov diffeomorphisms \cite{HY95,Hu00}.

Given $\al >0$, let $f_\al:[0,1]\to [0,1]$ be defined by
\begin{equation*}
f_\al(x)= \left\{
\begin{array}{cl}
x(1+2^{\alpha} x^{\alpha}) & \mbox{if}\; 0 \leq x \leq \frac{1}{2}  \\
2x-1 & \mbox{if}\; \frac{1}{2} < x \leq 1.
\end{array}
\right.
\end{equation*}
For $\alpha=0$ the map is expanding and (piecewise) $C^\infty$, while for $\alpha>0$ these maps have an indifferent fixed point at zero.
It is well known that $f_0$ preserves the Lebesgue measure $m$ and that, for every $\alpha \in (0,1)$, there exists a unique ergodic $f_\alpha$-invariant Sinai probability measure $\mu_\alpha \ll m$; see \cite{LSV99}. Lifting the dynamics to the circle, we obtain a diffeomorphism, so we may omit the discontinuity points and say that $f_\alpha \in C^{1+\alpha}$.

\begin{maincorollary}
The entropy map
$$ \alpha \in [0,1) \quad \mapsto \quad  h_{\mu_\alpha}(f_\alpha)$$
is continuous.
\end{maincorollary}

\begin{proof}
As a direct consequence of Theorem~\ref{thm:B} and Remark~\ref{rmk:uniformity} we deduce that, for any $\alpha_0>0$, the entropy map
$$\alpha \in [\alpha_0,1) \mapsto  h_{\mu_\alpha}(f_\alpha)$$
is continuous.

The continuity of the entropy map at $\alpha=0$ requires another argument since there is no $r>1$ such that $f_\alpha \in C^r$ for every $\alpha \in [0,1)$. Nevertheless, for any $\alpha \in [0,1)$
the map $f_\alpha$ is $C^{1+\alpha}$ and so, by the Pesin entropy formula,
$$h_{\mu_\alpha}(f_\alpha) = \int \log |f_\alpha'| \, d\mu_\alpha.$$
Moreover, the Jacobian map
$$\alpha \in [0,1) \mapsto \log |f_\alpha'| \in C^0([0,1])$$
is continuous. Finally, the continuity of the Sinai probability measure $(\mu_\alpha)_{\alpha \in [0,1)}$ follows from its uniqueness together with the existence of a generating partition for
the family $(f_\alpha)_{\alpha \,\in \,[0,1]}$, as explained in Remark~\ref{hyp.times}.
\end{proof}

Conditions for the existence of Sinai probability measures for diffeomorphisms that present an intermittency phenomenon, and for which our results apply, may be found in \cite{HY95, Hu00, H06}.
An argument entirely analogous to the previous ones yields the continuity of the entropy of the Sinai measure in this context as well.

\subsection{Conservative diffeomorphisms derived from the standard map}
In \cite{BC14} it has been introduced a $C^2$ open class $\mathcal U$ of volume-preserving non-uniformly hyperbolic diffeomorphisms, obtained by perturbing a skew product of standard maps driven by a parameter over an Anosov diffeomorphism. These systems are partially hyperbolic on $\mathbb T^2 \times \mathbb T^2$ with a two-dimensional central direction and belong to a neighborhood of $f_0 : \mathbb T^2 \times \mathbb T^2 \to \mathbb T^2 \times \mathbb T^2$ given by
$$ f_0(z,w)=(s(z) + \pi_1 \circ A^N(w), A^{2N}(w))$$
where $s: \mathbb T^2 \to \mathbb T^2$ denotes the standard map, $\pi_1$ stands for the projection of $\mathbb R^2$ onto the first coordinate, $A$ is a linear Anosov automorphism and $N\in \mathbb N$ is chosen large enough
(see \cite[Theorem~1.1]{BC14} for details).

To address the continuity of the metric entropy of the Lebesgue measure, which is a Sinai probability measure for this family of dynamics, and apply Theorem~\ref{thm:A}, we have to require uniqueness of the Sinai probability measure. By Pesin's ergodic decomposition theorem for non-uniformly hyperbolic measures invariant by $C^2$ volume-preserving diffeomorphisms (see \cite{Pe77}) the uniqueness is a direct consequence of the ergodicity of the Sinai probability measure. Moreover, as stated in \cite{BW10}, stable ergodicity for $C^2$ partially hyperbolic diffeomorphisms follows from accessibility and center bunching conditions; and stable ergodicity is a $C^1$ open and dense property in the space of partially hyperbolic diffeomoprhisms in $\text{Diff}^{\,\,r}(M)$ ($r\ge 1$).

Now, by \cite[Section~7]{BC14}, every diffeomorphism in $\mathcal U$ is partially hyperbolic and center bunched. In conclusion, it follows from the previously mentioned results that there is a $C^2$ open and dense set $\mathcal V \subset \mathcal U$ of (stable) ergodic diffeomorphisms. Additionally, $\mathcal V$ is a Baire space of $C^2$ diffeomorphisms. Thus, the following is an immediate consequence of Theorem~\ref{thm:A}:

\begin{maincorollary}
There is a $C^2$ residual subset $\mathcal R \subset \mathcal V$ of continuity points of the metric entropy function $f \mapsto h_m(f)$.
\end{maincorollary}

\subsection{Conservative diffeomorphisms with a dominated splitting}
In \cite{Ta04} the author presents open sets $\mathcal U$ of $C^2$ volume-preserving diffeomorphisms on the torus $\mathbb T^4$ which are not partially hyperbolic but are stably ergodic. These diffeomorphisms have a dominated splitting $T \,\mathbb T^4 = E^{cs} \oplus E^{cu}$ for which the Lebesgue measure is hyperbolic, exhibiting two positive and two negative Lyapunov exponents. Besides, the entropy map of the Lebesgue measure is continuous at a residual subset of these maps \cite{YZ15}. Now, Theorem~\ref{thm:B} improves this conclusion:

\begin{maincorollary}
The entropy map $f \in \mathcal U \mapsto h_m(f)$ is continuous.
\end{maincorollary}

\bibliographystyle{alpha}
\bibliography{bib}

\end{document}